\documentclass[a4paper,12pt]{article}
\usepackage[T1]{fontenc}
\usepackage{amsmath}
\usepackage{amsfonts}
\usepackage{lmodern}
\usepackage{moreverb}
\usepackage{amsthm}
\usepackage{hyperref}
\usepackage[utf8]{inputenc}
\usepackage[a4paper,tmargin=2.5cm,bmargin=3cm,
rmargin=2cm,lmargin=2cm]{geometry}
\newtheorem{theorem}{Theorem}
\newtheorem{corollary}{Corollary}
\newtheorem{lemma}{Lemma}
\newtheorem{proposition}{Proposition}
\newtheorem{remark}{Remark}
\theoremstyle{definition}
\newtheorem{definition}{Definition}
\newtheorem{example}{Example}
\author{K. Ernest Bognini\corref{cor1}\fnref{label1}}

\begin{document}

\begin{center}
{\large \textbf{A new factorization of the generalized period-doubling sequences through kernel words and gaps sequences}}\\
\vspace*{1cm}
K. Ernest Bognini$^{a, 1}$\\
$^{a}$\footnotesize{\textit{Department of Mathematics and Computer science \\
Joseph KI-ZERBO University\\
Ouagadougou, Burkina Faso}\\
  \texttt{ernestk.bognini@yahoo.fr}}\\
03 BP 7021 Ouagadougou\\
\vspace*{0.5cm}
Hamdi Ammar$^{b}$\\
$^{b}$\footnotesize{\textit{Department of Mathematics\\
Sfax University}\\
\texttt{hamdi.ammar.lfm@hotmail.fr}
 }\\
BP 802 Tunisia\\
\footnote{Corresponding author: K. Ernest Bognini (ernestk.bognini@yahoo.fr)}
 \begin{abstract}
 \noindent
 \\[2mm]In this paper, we study some new factorizations of period-doubling sequences over a $k$-letter alphabet, where $k\geq 2$. First, we define the combinatorial and arithmetic properties of these sequences. Then, we define the kernel words of period-doubling sequences and demonstrate how to factorize a binary sequence using its kernel words. Next, we define gap sequences for period-doubling sequences and explore their relationship with kernel words. Lastly, we present a factorization of period-doubling sequences for $k\geq 3$ based on kernel words and gap sequences.
 {\textbf{Keywords:} Palindromic factor, kernel word, gap sequences, word factorization, period-doubling sequences.
\\[2mm]
 {\textbf{2020 Mathematics Subject Clasification:} 11B85, 37B10, 68Q45, 68R15.
}
}
 \end{abstract}
\end{center}

   \section{Introduction}
\label{sec1}
The binary period-doubling sequence $\textbf{P}_{2}$ is the unique fixed point of the substitution defined over $\mathcal{A}_{2}=\{0,1\}$ by:
 $$s_{2}:\ 0\mapsto 01,\ 1\mapsto 00.$$
This sequence originates from chaotic dynamics (see \cite{a1, a2, a3}). The name period-doubling of this sequence comes from the fact the fundamental block doubles with each step.

In references \cite{1,2}, the authors presented a natural generalization of the period-doubling sequences over $A_{k}= \{0, 1, \cdots, k-1\}$ noted $\textbf{P}_{k}$. The authors studied the Ziv-Lempel and Crochemore factorizations based on the minimality and maximality of factors that occur at least once or twice, respectively, in $\textbf{P}_{k}$. Additionally, they established a variant of these factorizations based on palindromic factors in $\textbf{P}_{k}$.

The notion of kernel words, introduced in \cite{3} as a new property of Tribonacci words, allowed the authors to study gap sequences and factorize the Tribonacci sequence as a product of kernel words. In \cite{a2}, the authors generalized this concept to the $k$-bonacci sequence. They obtained a factorization of the $k$-bonacci from the generalized kernel words of the latter. A recent study in \cite{4} generalized kernel words and their relation to gap sequences in the Tribonacci sequence over an infinite alphabet.
In light of these interesting studies, we are interested in the kernel words and gap sequences of certain infinite substitutive words, including the following: Tribonacci, $k$-bonacci, and generalized Tribonacci sequences over an infinite alphabet. This paper explores a new factorization of generalized period-doubling sequences $\textbf{P}_{k}$ through their kernel words and gap sequences.

Our paper is organized as follows:
Section 2 provides useful definitions and basic notation that are useful for the rest of the paper. Section 3 establishes the combinatorial and arithmetic properties of $\textbf{P}_{k}$. Section 4 first defines the kernel words of $\textbf{P}_{k}$, then establishes properties of these words using their prefixes. Next, we demonstrate that the binary sequence can be completely factored using the kernel words. Finally, we define gap sequences and determines the factorization of $\textbf{P}_{k}$ for $k\geq 3$ using gap sequences and kernel words.

\section{Background}
\label{sec2}

 Let $\mathcal{A}$ be a finite alphabet. The set of finite words over $\mathcal{A}$ is denoted by $\mathcal{A}^{*}$ and $\varepsilon$ represents the empty word. The set of nonempty finite words (resp. infinite words) over $\mathcal{A}$ is denoted by $\mathcal{A}^{+}$ (resp. $\mathcal{A}^{\omega}$). The set of all finite and infinite words over $\mathcal{A}$ is denoted by $\mathcal{A}^{\infty}$. Let $u\in \mathcal{A}^{\infty}$ and $v\in \mathcal{A}^{*}$. The word $v$ is called factor of $u$ if there exist $u_{1}\in \mathcal{A}^{*},\ u_{2}\in\mathcal{A}^{\infty}$ such that $u=u_{1}vu_{2}$. The factor $v$ is called prefix (resp. suffix) if $u_{1}$ (resp. $u_{2}$) is empty. Suppose that $u=a_{1}a_{2}\cdots a_{n}$ is a finite word with $a_{i}\in \mathcal{A}$, for $i=1,2,\cdots,n$. The word $\overline{u}=a_{n}a_{n-1}\cdots a_{1}$ is called the mirror of $u$. A word is a palindrome if it is equal to its mirror. For all $u\in \mathcal{A}^{*}$, $\vert u \vert$ designates the length of $u$. For $a\in \mathcal{A}$, we denote the inverse of $a$ by $a^{-1}$, one has: $aa^{-1}=a^{-1}a=\varepsilon$. If $u$ is a finite word over $\mathcal{A}$ beginning (resp. ending) with the letter $a$ then $a^{-1}u$ (resp. $ua^{-1}$) designates the word obtained from $u$ by deleting its first (resp. last) letter.\\

Let $u=a_{1}a_{2}\cdots a_{n}$ be a finite word. For any $i\leq j\leq n$, define $u[i,j]=a_{i}a_{i+1}\cdots a_{j-1}a_{j}$ as the factor of $u$ of length $j-i+1$, starting with the $i-$th letter and ending with the $j-$th letter.\\
A morphism over $\mathcal{A}^{*}$ is a map $f \ :\mathcal{A}^{*}\longrightarrow \mathcal{A}^{*}$ such that $f(uv)=f(u)f(v)$ for all $u, v\in \mathcal{A}^{*}$.
It is $k$-uniform, $k\in \mathbb{N}^{*}$, if $\vert f(a) \vert= k$, for any $a\in \mathcal{A}$. A morphism is said to be prolongable on $a$, if there exists $u\in \mathcal{A}^{+}$ such that $f(a)=au$.
In this case, the word $f^{n}(a)$ is a proper prefix of the word $f^{n+1}(a)$, for any integer $n$. The sequence $(f^{n}(a))_{n\geq 0}$ converges to a unique infinite word noted $f^{\infty}(a)$ and called a purely morphic word.

\section{Combinatorial and Arithmetical Properties of $\textbf{P}_{k}$}
\label{sec3}
Let $\mathcal{A}_{k}=\{0,\ 1,\ 2,\cdots, k-1\}$ be an alphabet with $k\geq 2$ letters. We define the $k$-period-doubling substitution of $\mathcal{A}^{\ast}_k$ on itself as follows:

 $$s_{k}(m)=\left\{\begin{array}{cc}
0(m+1),\ \textrm{if}\  m\in \mathcal{A}_{k-2}\\
00,\hspace*{1.2cm} \ \textrm{if}\ m=k-1.
\end{array} \right.$$
The generalized period-doubling sequence is the infinite word generated by $s_{k}$ from $0$. Thus:
\begin{align*}
\textbf{P}_{k}&=s^{\infty}_{k}(0)\\
&=01 02 0103 01020104 0102010301020105 \cdots.
\end{align*}

Note that for all $k\geq 3$,
$$\textbf{P}_{k} \equiv \textbf{P}_{2}[\mod (k-1)],\ \text{where}$$
 $$\textbf{P}_{2}=01 00 0101 01000100 01000101000101 010001010100\cdots .$$

For the rest of the work, let us set out the following:
$$W_{n}=s_{k}^{n}(0)\hspace*{1cm}\ \text{and}\hspace*{1cm} W_{n,m}=s_{k}^{n}(m)\ \text{for all}\ m\in \mathcal{A}_{k}-\{0\}.$$
The following list of words is provided for a given value of $n$.
\begin{align*}
W_{0}&=0\hspace*{4cm} W_{0,m}= m\\
W_{1}&=01\hspace*{3.8cm} W_{1,m}=\left\{ \begin{array}{ll}
0(m+1)\hspace*{0.3cm} \mbox{if}\ m\in [|1;k-2|]\\
00 \hspace*{2cm} \mbox{if}\ m=k-1
  \end{array}\right.\\
W_{2}&=01 02 \hspace*{3.5cm} W_{2,m}=\left\{  \begin{array}{ll}
010(m+2)\hspace*{0.3cm} \mbox{if}\ m\in [|2;k-3|]\\
0101 \hspace*{2cm} \mbox{if}\ m=k-2
  \end{array}\right.\\
W_{3}&=01020103\hspace*{2.6cm} W_{3,m}=\left\{ \begin{array}{ll}
0102010(m+3)\hspace*{0.3cm} \mbox{if}\ m\in [|3;k-4|]\\
01020102 \hspace*{2cm} \mbox{if}\ m=k-3
  \end{array}\right.\\
  &\vdots \hspace*{5cm} \vdots
\end{align*}

The mirror substitution of $s_{k}$ is defined as follows: 
$$\tilde{s}_{k}(m)=\textbf{E}_{k}(m)0,\ \textrm{for all}\ m\in \mathcal{A}_{k},$$

where $\textbf{E}_{k}$, the exchange application is defined over $\mathcal{A}_{k}$ by:
$$\textbf{E}_{k}(m)= \left\{\begin{array}{cc}
 m+1, \ \textrm{if}\ m\in [|0; k-2|]\\
0,\hspace*{0.7cm} \  \textrm{if}\ m=k-1.
\end{array} \right.$$
\begin{remark} The following holds for all $m\in\mathcal{A}_{k}$:
\begin{enumerate}
\item $s_{k}(m)=0\emph{\textbf{E}}_{k}(m)$.
\item $\tilde{s}_{k}(m)=\widetilde{s_{k}(m)}$.
\end{enumerate}
\end{remark}
\begin{lemma}\label{L1}
Let $k\geq 2$ be a fixed natural number. Then, for all $n\geq 2$, we have:
\begin{enumerate}
\item $W_{n}=\bigg(\displaystyle \prod_{l=n-1}^{0}W_{l}\bigg )(n)$.
\item $W_{n,m}=\left\{ \begin{array}{ll}
W_{n-1}W_{n-1,m+1}\hspace*{0.3cm} \mbox{if}\ m\in [|0;k-2|]\\
W_{n-1}^{2} \hspace*{2cm} \mbox{if}\ m=k-1
  \end{array}\right.$
\end{enumerate}
\end{lemma}
\begin{proof}
\begin{enumerate}
\item This can be proven by induction on the integer $n$. The equality is easily verified for $n=1$. given that the property holds for the rank $n\geq 1$ and show that it also holds for the rank $n+1$.

\begin{align*}
W_{n+1}&= s_{k}(W_{n})\\
&=s_{k}\bigg( \displaystyle \prod_{l=n-1}^{0}W_{l}\bigg)(n)\\
&= \displaystyle \prod_{l=n-1}^{0}s_{k}\bigg(W_{l}\bigg)s_{k}(n)\\
&= \displaystyle \prod_{l=n-1}^{0}s_{k}(s_{k}^{l}(0))(n+1)\\
&=\displaystyle \prod_{l=n-1}^{0}s_{k}^{l+1}(0)(n+1)\\
&= \displaystyle \prod_{l=n}^{1}s_{k}^{l}(0)(n+1)\\
&= \bigg(\displaystyle \prod_{l=n}^{1}W_{l}\bigg)(n+1).
\end{align*}

\item \begin{align*} 
W_{n,m}&=s_{k}^{n+1}(m)\\
&=s_{k}^{n}(s_{k}(m))\\
&=s_{k}^{n}(0\textbf{E}_{k}(m))\\
&=\left\{ \begin{array}{ll}
s_{k}^{n}(0(m+1))=W_{n}W_{n,m+1}\hspace*{0.3cm} \mbox{if}\ m\in [|0;k-2|]\\
s_{k}^{n}(00)=W_{n}W_{n}=W_{n}^{2} \hspace*{2cm} \mbox{if}\ m=k-1
  \end{array}\right.
\end{align*}

\end{enumerate}
\end{proof}
\begin{proposition}\label{P1}
For all natural integers $n$ and $l$, $W_{n}W_{l}$ is a factor of $\textbf{P}_{k}$.
\end{proposition}
\begin{proof}
Note that the image of any letter $m$ in $\mathcal{A}_{k}$ begins with $0$. Therefore, $W_{n}0$ is a factor of $\textbf{P}_{k}$. We will continue the proof according to the cases below.\\
\noindent\textbf{Case 1.} $n\geq l$. Then, the following holds:
\begin{align*}
W_{n}W_{l}&=s_{k}^{n}(0)s_{k}^{l}(0)\\
&=s_{k}^{l}(s_{k}^{n-l}(0))s_{k}^{l}(0)\\
&=s_{k}^{l}(s_{k}^{n-l}(0)0)\\
&=s_{k}^{l}(W_{n-l}0).
\end{align*}
Since $W_{n-l}0$ is a factor of $\textbf{P}_{k}$ and $s_{k}(\textbf{P}_{k})=\textbf{P}_{k}$. Therefore, $s_{k}(W_{n-l}0)$ is also a factor of $\textbf{P}_{k}$.

\noindent\textbf{Case 2.} $n<l$. Then $W_{n+k}W_{l}=W_{n+k-1}W_{n+k-2}\cdots \underline{W_{n}W_{l}}$. Therefore, $W_{n}W_{l}$ is a factor of $\textbf{P}_{k}$.
\end{proof}
\begin{theorem}\cite{1}
Let $v$ be a factor of $\textbf{\emph{P}}_{k}$ such that $\vert v \vert> 2$. Then, the following statements are equivalent:
\begin{enumerate}
\item $v$ is a palindromic factor.
\item $s_{k}(v)0$ is a palindromic factor.
\item $0^{-1}s_{k}(v)$ is a palindromic factor.
\end{enumerate}
\end{theorem}
Let us consider $(p_{n})_{n\geq 0}$ the sequence of finite words over $\mathcal{A}_{k}$ defined by:

\begin{equation} \label{E1}
p_{0}=\varepsilon \ \textrm{and}\  p_{n+1}=p_{n}\vartheta_{n}p_{n},\ \textrm{with}\ \vartheta_{n}=n\mod k, \ \textrm{for all}\ n\geq 0.
\end{equation}

\begin{remark} \label{C3} For all $n\geq 1$, $p_{n}$ is a palindromic prefix of $\textbf{P}_{k}$. Furthermore, $p_{n}=s_{k}(p_{n-1})0$.
\end{remark}

\begin{lemma}\cite{1} \label{C4}
The following properties hold for all $n\geq 0$:
\begin{enumerate}
\item $W_{n}=p_{n}\vartheta_{n}$, where $p_{n}$ is a palindrome and $\vartheta_{n}$ a letter such that:
$$\left\{\begin{array}{ll}
 p_{0}=\varepsilon\\
p_{n+1}=s_{k}(p_{n})0
\end{array} \right. \ \textrm{and}\ \vartheta_{n}=n\mod k.$$
\item $p_{n+1}=W_{n}W_{n-1}\cdots W_{1}W_{0}=\displaystyle\prod_{k=0}^{n}W_{n-k}$.
\end{enumerate}
\end{lemma}
\begin{theorem}\cite{2}\label{T} Let $v$ be a nonempty palindromic factor of $\textbf{P}_{k}$ such that $v\neq 00$. Then, we have:
\begin{enumerate}
\item If $v$ begins with an odd power of the letter $0$, then there exists a palindromic factor $v'$ of $\textbf{P}_{k}$ such that $v=s_{k}(v')0$.
\item If $v$ begins with an even power of the letter $0$, then there exists a palindromic factor $v'$ of $\textbf{P}_{k}$ such that $v=0^{-1}s_{k}(v')$.
\end{enumerate}
\end{theorem}
\begin{definition}
The finite $k$-period-doubling sequence $W_{m}$ is defined by he following relation:
$$W_{m} = W_{m-1}W_{m-2}W_{m-3}W_{m-4}\ldots W_{m-(k-1)}W_{m-k}W_{m-k} ~~ for ~~ all ~~ m\geq k.$$

The positive integer $w_{m} := |W_{m}|$ is called the $m$-th $k$-periode-doubling number defined as follow:

$$w_{m} = w_{m-1} + w_{m-2} + w_{m-3} + w_{m-4} \cdots + w_{m-(k-1)} + 2(w_{m-k}) ~~ for ~~ all ~~ m\geq k.$$

\end{definition}

\begin{theorem}\label{T1}
For $n\geq k-1$ and $i\in\{1,2,\ldots, k-1\}.$\\
The longest common prefix of $\displaystyle\prod_{j=n-(i+1)}^{n-k}W_{j}W_{n-1}\displaystyle\prod_{j=n-2}^{n-i}W_{j}$ and $W_{n}$ is $p_{n-i}.$\\
\end{theorem}
\begin{proof}
Using induction on $n$, we show that for $n\geq k-1,$ $\displaystyle\prod_{l=n-(i+1)}^{n-k}W_{l}W_{n-1}$ is not a prefix of $W_{n}$ and their longest common prefix is $p_{n-i}.$\\
We have for $n = k-1,$
\begin{align*}
  W_{n} &= W_{k-1} \\
   &= \bigg(\displaystyle \prod_{l=k-2}^{0}W_{l}\bigg )(k-1) ~~~~~ by ~~ Lemma ~~\ref{L1} \\
   &= \underline{W_{k-2}}\underbrace{W_{k-3}\ldots W_{2}W_{1}W_{0}}(k-1)  \\
   &= \underline{\bigg(\displaystyle \prod_{l=k-3}^{0}W_{l}\bigg )(k-2)}\underbrace{\bigg(\displaystyle \prod_{l=k-3}^{0}W_{l}\bigg )}(k-1) ~~~~~ by ~~ Lemma ~~\ref{L1}  \\
   &= \underline{W_{k-3}}\underbrace{W_{k-4}\ldots W_{2}W_{1}W_{0}}(k-2)\bigg(\displaystyle \prod_{l=k-3}^{0}W_{l}\bigg )(k-1) \\
   &= \underline{\bigg(\displaystyle \prod_{l=k-4}^{0}W_{l}\bigg )(k-3)}\underbrace{\bigg(\displaystyle \prod_{l=k-4}^{0}W_{l}\bigg )}(k-2) \bigg(\displaystyle \prod_{l=k-3}^{0}W_{l}\bigg)(k-1) ~~~~~ by ~~ Lemma ~~\ref{L1}  \\
   &\vdots  \\
   &= \bigg(\displaystyle \prod_{l=k-i-2}^{0}W_{l}\bigg )\overbrace{(k-i-1)}\bigg(\displaystyle \prod_{l=k-i-2}^{0}W_{l}\bigg )(k-i)\bigg(\displaystyle \prod_{l=k-i-1}^{0}W_{l}\bigg )(k-i+1)\\
   &\bigg(\displaystyle \prod_{l=k-i}^{0}W_{l}\bigg )(k-i+2)\ldots \bigg(\displaystyle \prod_{l=k-3}^{0}W_{l}\bigg )(k-1) ~~~~~ by ~~ Lemma ~~\ref{L1} ,
\end{align*}
and \\
\begin{align*}
  \displaystyle \prod_{l=n-(i+1)}^{n-k}W_{l}W_{n-1} &= \displaystyle \prod_{l=k-1-(i+1)}^{-1}W_{l}W_{k-2} \\
    &= \underbrace{\displaystyle \prod_{l=k-i-2}^{0}}W_{l}\overbrace{W_{-1}}W_{k-2}.
\end{align*}
So, $\underbrace{\displaystyle \prod_{l=k-i-2}^{0}}W_{l}\overbrace{W_{-1}}W_{k-2}$ is not a prefix of $W_{k-1}$ since $W_{-1} \neq k-i-1$ and
the longest common prefix of $\displaystyle \prod_{l=k-i-2}^{0}W_{l}W_{-1}W_{k-2}$ and $W_{k-1}$ is $$\displaystyle \prod_{l=k-i-2}^{0}W_{l} = p_{k-1-i}.$$
Assume that for $n\geq k-1.$
Suppose the statement about the longest common prefix holds for all $n$ with $k-1\leq n\leq m,$ and we now prove it for $n = m+1.$\\
We have $$\displaystyle \prod_{l=m+1-(i+1)}^{m+1-k}W_{l}W_{m+1-1} = \prod_{l=m-i}^{m+1-k}W_{l}W_{m}$$ and $$W_{n} = W_{m+1}.$$
By the induction hypothesis, for all $n$ with $k-1\leq n\leq m,$ the longest common prefix of $\displaystyle \prod_{l=m-(i+1)}^{m-k}W_{l}W_{m-1}$ and $W_{m}$ is $p_{m-i}.$\\
Since $$W_{m-1} ~~ is ~~ a ~~ prefix ~~ of ~~ W_{m},$$
$$W_{m} ~~ is ~~ a ~~ prefix ~~ of ~~ W_{m+1}$$
and we have $$\displaystyle \prod_{l=m-(i+1)}^{m-k}W_{l} ~~ is ~~ a ~~ prefix ~~ of ~~ \displaystyle \prod_{l=m-i}^{m+1-k}W_{l}$$
then, the longest common prefix of $\displaystyle \prod_{l=m-i}^{m+1-k}W_{l}W_{m}$ and $W_{m+1}$ is $$\displaystyle \prod_{l=m-i}^{m+1-k}W_{l}p_{m-i} = p_{m-i+1}.$$
Thus, by induction on $n,$ the result holds.
\end{proof}
\begin{theorem}
The period-doubling sequence is defined by the substitution $s_{k}.$
$$\textbf{P}_{k} = \displaystyle \prod_{i=0}^{\infty} \overline{W_{i}} = \overline{W_{0}} ~ \overline{W_{1}} ~ \overline{W_{2}} \cdots.$$
\end{theorem}
\begin{proof}
The proof of this Theorem follows from, Lemma \ref{C4}, the definition of the sequence $p_{n}$ and Theorem \ref{T1}.\\
From Lemma \ref{C4} and the definition of $p_{n},$ we have:
$$p_{n} = \overline{W_{0}} ~ \overline{W_{1}} ~ \overline{W_{2}}\cdots \overline{W_{n-1}}.$$
According Theorem \ref{T1}, for any $n,$ $p_{n}$ is a prefix of $W_{n+1}.$ Since $W_{n+1}$ itself is a prefix of $\textbf{P}_{k},$ it follows that $p_{n}$ being a palindrome is also a prefix of $\textbf{P}_{k}.$\\
 Hence,
$$\textbf{P}_{k} = \displaystyle \prod_{i=0}^{\infty} \overline{W_{i}} = \overline{W_{0}} ~ \overline{W_{1}} ~ \overline{W_{2}} \cdots.$$
\end{proof}

\section{New factorization of $\textbf{P}_k$ using kernel words and gap sequences}\label{sec4}
In this section, we frist generalize the definition of the kernel word associated with the period-doubling sequence and study some of its properties. Next, we provide explicit formulas for factorizing the binary period-doubling sequence into kernel words. Lastly, we define the gap sequences with respect to the kernel words in the generalized period-doubling sequences and provide their factorization.

\subsection{Kernel words and some of its related properties}\label{subsec4.1}
The terms \emph{kernel number} and \emph{kernel word} are defined as follows:

\begin{definition}(Kernel number of the period-doubling number)\label{17}\\
 Let $\{r_{i}\}_{i\geq0}$ be the sequence of positive integers defined by:

$$\left\{
  \begin{array}{ll}
    r_{0} = 0, \\
    r_{1} = r_{2} = r_{3} = \cdots = r_{k-1} = r_{k} = 1, \\
    r_{i} = r_{i-1} + r_{i-2} + r_{i-3} + \cdots + r_{i-(k-1)} + 2(r_{i-k}) - (k-2) \ \ for \ \ m> k.
  \end{array}\label{e}(1)
\right.$$
The number $r_{i}$ is called the $i$-th kernel number of the period-doubling number.
\end{definition}
\begin{definition}(Kernel words of the period-doubling sequence)\label{19}\\
Let $\{R_{i}\}_{i\geq 0}$ be the sequence of factors of $\textbf{P}_{k}$ defined as follows:
$$\left\{
    \begin{array}{ll}
      R_{0} = \epsilon \\
      R_{i} = (i-1), ~~ \forall ~~ i\in[|1,k|], \\
      R_{k+1} = 000, \\
      R_{k+2} = 10101,
    \end{array}
  \right.$$
and
$$R_{i} = \left\{
            \begin{array}{ll}
              0^{-1}s_{k}(R_{i-1}) ~~~~~~~~~~ if ~~ i ~~ is ~~ even,  \\
~~~~~~~~~~~~~~~~~~~~~~~~~~~~~~~~~~~~~~~~~~~~~~~~~~~~~~~~~~~~~~~~\forall ~~ i> k+2 \\
              s_{k}(R_{i-1})0 ~~~~~~~~~~~~~~ otherwise
            \end{array}
          \right.$$
The factors $R_{i}$ for all $i\geq 0$ are called the kernel words of $\textbf{P}_{k}$ with order $i.$
\end{definition}
\begin{example}
The first few values, kernel numbers, and kernel words of the $3$-period-doubling sequence as follows:
$$\textbf{P}_{3} = 0102010001020101010201000102010201020100010201010102010001020100010201000\ldots$$
$$\left\{
  \begin{array}{ll}
    R_{0} = \epsilon \\
    R_{1} = 0 \\
    R_{2} = 1 \\
    R_{3} = 2 \\
    R_{4} = 000 \\
    R_{5} = 10101 \\
    R_{6} = 201020102 \\
    R_{7} = 0001020100010201000 \\
    \vdots
  \end{array}
\right. ~~ and ~~ \left\{
                    \begin{array}{ll}
                      r_{0} = 0 \\
                      r_{1} = 1 \\
                      r_{2} = 1 \\
                      r_{3} = 1 \\
                      r_{4} = 3 \\
                      r_{5} = 5 \\
                      r_{6} = 9 \\
                      r_{7} = 19 \\
                      \vdots
                    \end{array}
                  \right.$$
\end{example}

 \subsection{Factorization of $\textbf{\emph{P}}_{k},\ k\geq 2$ with respect to gap sequences and kernel words}\label{subsec4.2}
\subsubsection{Factorization of $\textbf{P}_{2}$ into kernel words}\label{subsubsec4.2.1}
\begin{proposition}
The binary period-doubling sequence can be uniquely factored into kernel words as follows:
$$\textbf{P}_{2} = \displaystyle \prod_{i\geq 1} R_{i} = 0 ~~ 1 ~~ 000 ~~ 10101 ~~ 00010001000\ldots$$
\end{proposition}
\begin{proof}
\begin{itemize}
  \item For $i\in [|0,4|],$ we have $R_{0} = \epsilon,$ $R_{1} = 0,$ $R_{2} = 1,$ $R_{3} = 000,$ $R_{4} = 10101.$ They are palindromic factors of $\textbf{P}_{2}.$\\
\item For $i\geq 5$ use the above Definition and Theorem \ref{T}, we have each $R_{i}$ is a palindromic factors of $\textbf{P}_{2}.$\\
\item For the uniqueness of each factor $R_{i}$  in the factorization, we show that $R_{i}$ is not a factor of $R_{i+1}$ for all $i\geq 0.$ (See, the above Definition).

\end{itemize}

\end{proof}

\subsubsection{Factorization of $\textbf{P}_{k},\ k\geq 3$ using kernel words and gap sequences}\label{subsubsec4.2.2}

We will now define the specific factors in $\textbf{P}_{k}$ that will be useful for the rest of the work. 
\begin{definition}
Let $w$ be a factor of the $k$-period-doubling sequence $\textbf{P}_{k} = u_{1}u_{2}\cdots$. Then, $w$ occurs infinitely many times which in the sequence, which we arrange by the sequence $\{w_{p}\}_{p\geq1},$ where $w_{p}$ denotes the $p$-th occurrence of $w.$
\end{definition}
Now, we will define gap sequences and the set of gaps for an infinite sequence.
 \begin{definition}
Let $G_{p}(w)$ be the gap between $w_{p}$ and $w_{p+1},$ we call $\{G_{p}(w)\}_{p\geq1}$ the gap sequence of the factor $w.$
\end{definition}
\begin{definition}
The set of gaps of factor $w$ is defined as follows: $$\{G_{p}(w)\ \mid p\geq1\}.$$
\end{definition}

Let $w$ be of length $n$, and let $w_p = u_{i+1}u_{i+2}\cdots u_{i+n}$ and $w_{p+1} = u_{j+1}u_{j+2}\cdots u_{j+n}$. The gap between $w_p$ and $w_{p+1}$ is defined by: 
$$G_{p}(w) = \left\{
               \begin{array}{ll}
                 \epsilon ~~~~~~~~~~~~~~~~~~~~~~~~~~~~~~~~~ when ~~ i+n = j, ~~ w_{p} ~~ and ~~ w_{p+1} ~~ are ~~ adjacent;  \\
                 u_{i+n+1}\cdots u_{j} ~~~~~~~~~~~~~~~~~ when ~~ i+n < j, ~~ w_{p} ~~ and ~~ w_{p+1} ~~ are ~~ separated; \\
                 (u_{j+1}\cdots u_{i+n})^{-1}~~~~~~~~~~~~ when ~~ i+n > j, ~~ w_{p} ~~ and ~~ w_{p+1} ~~ are ~~ overlapped.
               \end{array}
             \right.$$
The sequence $\{G_{p}(w)\}_{p\geq1}$ is called the gap sequence of the factor $w.$ By convention, we define $G_{0}(w)$ as the prefix of $\textbf{P}_{k}$ before $w_{1}.$ When $w_{p}$ and $w_{p+1}$ overlap, the overlapping part is the word $u_{j+1}\cdots u_{i+n}$. We take its inverse word as the gap $G_{p}(w).$
\begin{definition}
The gap sequences of the $k$-period-doubling sequences for $k\geq 3$ are given by the following:
$$ \left\{\begin{array}{ll}
               G_{n}=p_{n}\ \forall\ n\in [|0, k-2|]\\
               G_{n}=\left\{
               \begin{array}{ll}
               s_{k}(G_{n-1})\ if\ n\not\equiv 0 \mod k\\
               \hspace*{6cm}\forall \ n\in [|k-1, k|]\\
               \tilde{s}_{k}(G_{n-1}) \ otherwise
               \end {array}
             \right.\\
               G_{n}=s_{k}(G_{n-1})0\ \forall\ n\geq k+1.
              \end {array}
             \right.
             $$
\end{definition}
\begin{example}
For $k = 3,$ we illustrate the gaps between two consecutive kernel words. We denote the gap between the kernel words $R_{n}$ and $R_{n+1}$ by $G_{n},$ for all $n\geq 1.$
\begin{eqnarray*}
  \textbf{P}_{3} &=& \underline{0} ~~~~~ \underbrace{\epsilon} ~~~~ \underline{1} ~~~~~ \underbrace{0} ~~~~~ \underline{2} ~~~~ \underbrace{01} ~~ \underline{000} ~~~~ \underbrace{1020} ~~~~~~ \underline{10101} ~~ \underbrace{020100010}\ldots  \\
   &=& \underline{R_{1}}~~~\underbrace{G_{1}}~~\underline{R_{2}} ~~~ \underbrace{G_{2}} ~~~~ \underline{R_{3}} ~~~ \underbrace{G_{3}} ~~ \underline{R_{4}} ~~~~~ \underbrace{G_{4}} ~~~~~~~~~~ \underline{R_{5}} ~~~~~~~~~~ \underbrace{G_{5}}\ldots
\end{eqnarray*}
\end{example}
\begin{remark}
The gap sequences for the binary period-doubling sequence are an empty word for any integer. Indeed,
$$G_{n}=\varepsilon,\ \text{for all}\ n\geq 0.$$
\end{remark}

\begin{definition}
Let $\{R_{n}\}_ {n\geq 1}$ be the sequence of factors defined as follows:
$$\left\{
  \begin{array}{ll}
    R_{n} = R_{n}[1,r_{n-1}+1]W_{n-k}[1,r_{n-1}] \hspace*{1cm} if \ n\equiv 1 \pmod{k}\\
    R_{n} = R_{n}[1,r_{n-1}]W_{n-k}[1,r_{n-1}-1] \hspace*{1.4cm} otherwise

  \end{array}
\right.$$
\end{definition}

\begin{theorem} \label{t4.2}
For $k\geq 3$ and for any $n\geq k+1$, we have:

$$\left\{
  \begin{array}{ll}
    R_{n} = R_{n}[1,r_{n-1}+1]s_{k}^{n-(k+1)}(0)R_{n-k}[r_{n-(k+1)}+2,r_{n-k}] \hspace*{0.9cm} if \ n\equiv 1 \pmod{k}\\
    R_{n} = R_{n}[1,r_{n-1}]s_{k}^{n-(k+1)}(0)R_{n-k}[r_{n-(k+1)}+1, r_{n-k}] \hspace*{1.5cm} otherwise

  \end{array}
\right.$$
\end{theorem}
\begin{proof}
First, we prove the formula for $n \not\equiv \ 1 \ (\mathrm{mod}\ k).$\\
According to the definition of the kernel word, we have:
\begin{eqnarray*}
  R_{n} &=& R_{n}[1,r_{n-1}]W_{n-k}[1,r_{n-1}-1] \\
   &=& R_{n}[1,r_{n-1}]s_{k}^{n-k}(0)[1,r_{n-1}-1] \\
   &=& R_{n}[1,r_{n-1}]\Bigg(s_{k}^{n-(k+1)}(0)s_{k}^{n-(k+1)}(1)\Bigg)[1,r_{n-1}-1] \\
   &=& R_{n}[1,r_{n-1}]s_{k}^{n-(k+1)}(0)s_{k}^{n-(k+1)}(1)[1,r_{n-1}-w_{n-(k+1)}-1] ~~~~ since ~~ r_{n}> w_{n-(k+1)} \\
   &=& R_{n}[1,r_{n-1}]s_{k}^{n-(k+1)}(0)R_{n-k}[1,r_{n-(k+1)}]^{-1}R_{n-k}[1,r_{n-(k+1)}]\\
&&\Bigg(s_{k}^{n-(k+2)}(0)s_{k}^{n-(k+2)}(2)\Bigg)[1,r_{n-1}-w_{n-(k+1)}-1] \\
   &=& R_{n}[1,r_{n-1}]s_{k}^{n-(k+1)}(0)R_{n-k}[1,r_{n-(k+1)}]^{-1}R_{n-k}[1,r_{n-(k+1)}]\\
&&\Bigg(s_{k}^{n-(k+2)}(0)s_{k}^{n-(k+2)}(2)\Bigg)[1,r_{n-(k+1)}-1] ~~~~ since ~~ r_{n-(k+1)} = r_{n-1} w_{n-(k+1)} \\
   &=& R_{n}[1,r_{n-1}]s_{k}^{n-(k+1)}(0)R_{n-k}[1,r_{n-(k+1)}]^{-1}\Bigg(R_{n-k}[1,r_{n-(k+1)}]
s_{k}^{n-(k+2)}(0)[1,r_{n-(k+1)}-1]\Bigg)\\
&& since ~~ w_{n-(k+2)}> r_{n-(k+1)} \\
   &=&  R_{n}[1,r_{n-1}]s_{k}^{n-(k+1)}(0)R_{n-k}[1,r_{n-(k+1)}]^{-1}\\
&&\Bigg(R_{n-k}[1,r_{n-(k+1)}]
\Bigg(s_{k}^{n-(k+3)}(0)s_{k}^{n-(k+3)}(1)\Bigg)[1,r_{n-(k+1)}-1]\Bigg) 
\end{eqnarray*}

\begin{eqnarray*}
   R_{n}&=& R_{n}[1,r_{n-1}]s_{k}^{n-(k+1)}(0)R_{n-k}[1,r_{n-(k+1)}]^{-1}\Bigg(R_{n-k}[1,r_{n-(k+1)}]
s_{k}^{n-(k+3)}(0)[1,r_{n-(k+1)}-1]\Bigg)\\
&& since ~~ w_{n-(k+3)}> r_{n-(k+1)} \\
   &=& R_{n}[1,r_{n-1}]s_{k}^{n-(k+1)}(0)R_{n-k}[1,r_{n-(k+1)}]^{-1}\\
&&\Bigg(R_{n-k}[1,r_{n-(k+1)}]
\Bigg(s_{k}^{n-(k+4)}(0)s_{k}^{n-(k+4)}(1)\Bigg)[1,r_{n-(k+1)}-1]\Bigg) \\
   &=& R_{n}[1,r_{n-1}]s_{k}^{n-(k+1)}(0)R_{n-k}[1,r_{n-(k+1)}]^{-1}\Bigg(R_{n-k}[1,r_{n-(k+1)}]
s_{k}^{n-(k+4)}(0)[1,r_{n-(k+1)}-1]\Bigg)\\
&& since ~~ w_{n-(k+4)}> r_{n-(k+1)} \\
   &=& R_{n}[1,r_{n-1}]s_{k}^{n-(k+1)}(0)R_{n-k}[1,r_{n-(k+1)}]^{-1}\\
&&\Bigg(R_{n-k}[1,r_{n-(k+1)}]
\Bigg(s_{k}^{n-(2k)}(0)s_{k}^{n-(2k)}(1)\Bigg)[1,r_{n-(k+1)}-1]\Bigg) \\
   &=& R_{n}[1,r_{n-1}]s_{k}^{n-(k+1)}(0)R_{n-k}[1,r_{n-(k+1)}]^{-1}\underline{\Bigg(R_{n-k}[1,r_{n-(k+1)}]
s_{k}^{n-(2k)}(0)[1,r_{n-(k+1)}-1]\Bigg)}\\
&& since ~~ w_{n-(2k)}> r_{n-(k+1)} \\
   &=& R_{n}[1,r_{n-1}]s_{k}^{n-(k+1)}(0)R_{n-k}[1,r_{n-(k+1)}]^{-1} \underline{R_{n-k}} \\
   &=& R_{n}[1,r_{n-1}]s_{k}^{n-(k+1)}(0)R_{n-k}[r_{n-(k+1)}+1, r_{n-k}]
\end{eqnarray*}
An analogous argument yields thecase $n\equiv 1 \pmod{k}$.
\end{proof}
\begin{lemma} \label{l4.1}
The following equalities hold for all integers positive integers $n$.
$$\left\{
  \begin{array}{ll}
    W_{n} = R_{1}G_{1}R_{2}G_{2}\ldots R_{n}G_{n}R_{n+1}[1, r_{n}+1] \hspace*{0.9cm} if \ n\equiv 0 \pmod{k}\\
    W_{n} = R_{1}G_{1}R_{2}G_{2}\ldots R_{n}G_{n}R_{n+1}[1, r_{n}] \hspace*{1.5cm} otherwise

  \end{array}(2)
\right.
$$
and
$$\left\{
  \begin{array}{ll}
    W_{n,1} = R_{n+1}[1, r_{n}+1]^{-1}R_{n+1}G_{n+1}R_{n+2}[1, r_{n+1}] \hspace*{0.9cm} if \ n\equiv 0 \pmod{k}\\
    W_{n,1} = R_{n+1}[1, r_{n}]^{-1}R_{n+1}G_{n+1}R_{n+2}[1, r_{n+1}] \hspace*{1.5cm} otherwise

  \end{array}(3)
\right.
$$

\end{lemma}
\begin{proof}
We will start by proving the formula for $n \not\equiv \ 0 \ (\mathrm{mod}\ k).$\\
The case $n= 1$ can be verified directly.\\
Suppose equations $(2)$ and $(3)$ are valid for all integers $m\leq n.$\\
Then, for $m=n+1,$ we obtain:
\begin{eqnarray*}
  W_{n+1} &=& s_{k}^{(n+1)}(0) \\
   &=& s_{k}^{n}(s_{k}(0)) \\
   &=& s_{k}^{n}(01) \\
   &=& s_{k}^{n}(0)s_{k}^{n}(1)\\
   &=& W_{n}W_{n,1}.
\end{eqnarray*}
According to the induction hypothesis, it follows that:
$$W_{n} = R_{1}G_{1}R_{2}G_{2}\ldots R_{n}G_{n}R_{n+1}[1, r_{n}] ~~ and ~~ W_{n,1} = R_{n+1}[1, r_{n}]^{-1}R_{n+1}G_{n+1}R_{n+2}[1, r_{n+1}].$$
Substituting these identities yields:
\begin{eqnarray*}
    W_{n+1} &=& W_{n}W_{n,1} \\
   &=& R_{1}G_{1}R_{2}G_{2}\ldots R_{n}G_{n}\underbrace{R_{n+1}[1, r_{n}]R_{n+1}[1, r_{n}]^{-1}}R_{n+1}G_{n+1}R_{n+2}[1, r_{n+1}] \\
   &=& R_{1}G_{1}R_{2}G_{2}\ldots R_{n}G_{n}R_{n+1}G_{n+1}R_{n+2}[1, r_{n+1}].
\end{eqnarray*}
Therefore, equation $(2)$ is also verified for $n+1$ as well.\\
From the preceding step, it follows that by Lemma \ref{C4}, we have:
  \begin{eqnarray*}
  W_{n+2} &=& p_{n+2} \vartheta_{n+2} \\
          &=& \underline{p_{n+1} \vartheta_{n+1}} \underbrace{p_{n+1} \vartheta_{n+2}} \\
          &=& \underline{W_{n+1}} \underbrace{W_{n+1,1}} \\
   &=& \Bigg(R_{1}G_{1}R_{2}G_{2}\ldots R_{n+1}G_{n+1}R_{n+2}[1, r_{n+1}+1]\Bigg)\\
&&\Bigg(R_{n+2}[1, r_{n+1}]^{-1}R_{n+2}G_{n+2}R_{n+3}[1, r_{n+2}]\Bigg) \\
   &=& R_{1}G_{1}R_{2}G_{2}\ldots R_{n+1}G_{n+1}R_{n+2}G_{n+2}R_{n+3}[1, r_{n+2}].
\end{eqnarray*}
As an alternative, we can apply the substitution rule to obtain:
\begin{eqnarray*}
  W_{n+2} &=& s_{k}^{n+2}(0) = s_{k}^{n+1}(s_{k}(0)) = s_{k}^{n+1}(01) = s_{k}^{n+1}(0)s_{k}^{n+1}(1) \\
   &=& R_{1}G_{1}R_{2}G_{2}\ldots R_{n+1}G_{n+1}R_{n+2}[1, r_{n+1}]s_{k}^{n+1}(1).
\end{eqnarray*}
A comparison of equations $(5)$ and $(6)$ yields the following conclusion:
\begin{equation*}
    W_{n+2,1} = s_{k}^{n+2}(1) = R_{n+2}[1, r_{n+1}]^{-1}R_{n+2}G_{n+2}R_{n+3}[1, r_{n+2}].
\end{equation*}
This proves equation $(6)$ for $n+1.$\\
Therefore, both identities $(2)$ and $(3)$ are valid for all $n\geq 1,$ which completes the proof.\\
In the case of $(n\equiv 0 \mod k),$ the constructions of $W_{n}$ and $W_{n,1}$ follow from an analogous reasoning.
\end{proof}
\begin{proposition} \label{p4.2}
For every $n\geq k+1,$ the kernel word can be expressed as follows:
$$\left\{
  \begin{array}{ll}
    R_{n} = R_{n}[1,r_{n-1}+1]R_{1}G_{1}R_{2}G_{2}\ldots R_{n-(k+1)}G_{n-(k+1)}R_{n-k} \hspace*{0.9cm} if \ n\equiv 1 \pmod{k}\\
    R_{n} = R_{n}[1,r_{n-1}]R_{1}G_{1}R_{2}G_{2}\ldots R_{n-(k+1)}G_{n-(k+1)}R_{n-k} \hspace*{1.5cm} otherwise

  \end{array}
\right.$$
\end{proposition}
\begin{proof}
By applying Theorem \ref{t4.2} together with equatuion \eqref{l4.1}, we obtain: \\
$\textbf{Case 1:}$ If $(n-1)\equiv 1 \pmod{k},$ we have:
\begin{eqnarray*}
  R_{n} &=& R_{n}[1,r_{n-1}+1]s_{k}^{n-(k+1)}(0)R_{n-k}[r_{n-(k+1)}+2,r_{n-k}] \\
   &=& R_{n}[1,r_{n-1}+1]\underline{s_{k}^{n-(k+1)}(0)}R_{n-k}[1,r_{n-(k+1)}+1]^{-1}R_{n-k}[1,r_{n-k}]  \\
   &=& R_{n}[1,r_{n-1}+1]\underline{R_{1}G_{1}R_{2}G_{2}\ldots R_{n-(k+1)}G_{n-(k+1)}R_{n-k}[1,r_{n-(k+1)}+1]}R_{n-k}[1,r_{n-(k+1)}+1]^{-1} \\ && R_{n-k}[1,r_{n-k}] \\
   &=& R_{n}[1,r_{n-1}+1]R_{1}G_{1}R_{2}G_{2}\ldots R_{n-(k+1)}G_{n-(k+1)}R_{n-k}[1,r_{n-k}] \\
   &=& R_{n}[1,r_{n-1}+1]R_{1}G_{1}R_{2}G_{2}\ldots R_{n-(k+1)}G_{n-(k+1)}R_{n-k}
\end{eqnarray*}
$\textbf{Case 2:}$ If $(n-1)\not\equiv 1 \pmod{k},$ we have:
\begin{eqnarray*}
  R_{n} &=& R_{n}[1,r_{n-1}]s_{k}^{n-(k+1)}(0)R_{n-k}[r_{n-(k+1)}+1, r_{n-k}] \\
   &=& R_{n}[1,r_{n-1}]\underline{s_{k}^{n-(k+1)}(0)}R_{n-k}[1,r_{n-(k+1)}]^{-1}R_{n-k}[1,r_{n-k}]  \\
   &=& R_{n}[1,r_{n-1}]\underline{R_{1}G_{1}R_{2}G_{2}\ldots R_{n-(k+1)}G_{n-(k+1)}R_{n-k}[1,r_{n-(k+1)}]}R_{n-k}[1,r_{n-(k+1)}]^{-1} \\ && R_{n-k}[1,r_{n-k}] \\
   &=& R_{n}[1,r_{n-1}]R_{1}G_{1}R_{2}G_{2}\ldots R_{n-(k+1)}G_{n-(k+1)}R_{n-k}[1,r_{n-k}] \\
   &=& R_{n}[1,r_{n-1}]R_{1}G_{1}R_{2}G_{2}\ldots R_{n-(k+1)}G_{n-(k+1)}R_{n-k}
\end{eqnarray*}
\end{proof}
\begin{theorem}
For $n\geq k+1,$ the sequence of gaps $G_{n}$ occurring between two successive kernel words $R_{n}$ and $R_{n+1}$ is described by the following formula:
$$\left\{
  \begin{array}{ll}
    G_{n} = G_{n-k}\Bigg(\displaystyle \prod_{i=k-1}^{2}(R_{n-i}G_{n-i})\Bigg)R_{n-1}[1,r_{n-2}+1]s_{k}^{n-2}(2)\ast \hspace*{0.9cm} if \ (n-1)\equiv 1 \pmod{k}\\
    G_{n} = G_{n-k}\Bigg(\displaystyle \prod_{i=k-1}^{2}(R_{n-i}G_{n-i})\Bigg)R_{n-1}[1,r_{n-2}]s_{k}^{n-2}(2)\ast \hspace*{1.5cm} otherwise

  \end{array}
\right.$$
where $\ast$ is a suffix of $s_{k}^{n-2}(0)$ and is defined as follows:
$$\left\{
  \begin{array}{ll}
    R_{n+1}[1,r_{n}+1]^{-1} \hspace*{0.9cm} if \ n\equiv 0 \pmod{k}\\
    R_{n+1}[1,r_{n}]^{-1} \hspace*{1.5cm} otherwise

  \end{array}
\right.$$
\end{theorem}
\begin{proof}
According to the iteration of $s_{k},$ using Lemma \ref{l4.1} and Proposition \ref{p4.2}, we get:
\begin{eqnarray*}
  W_{n} &=& s_{k}^{n}(0) = s_{k}^{n-1}(01) = s_{k}^{n-1}(0)s_{k}^{n-1}(1) \\
   &=& s_{k}^{n-1}(0)s_{k}^{n-2}(02) \\
   &=& s_{k}^{n-1}(0)s_{k}^{n-2}(0)s_{k}^{n-2}(2).
\end{eqnarray*}
We distinguish two cases:\\
$\textbf{Case 1:}$ If $(n-1)\equiv 1 \pmod{k},$ then, $(n-2)\equiv 0 \pmod{k}$ and $(n-1) \not\equiv 0 \pmod{k}.$\\
Therefore, $$W_{n-1} = R_{1}G_{1}R_{2}G_{2}\ldots R_{n-1}G_{n-1}R_{n}[1, r_{n-1}]$$ and
 $$W_{n-2} = R_{1}G_{1}R_{2}G_{2}\ldots R_{n-2}G_{n-2}R_{n-1}[1, r_{n-2}+1].$$
Then:
\begin{eqnarray*}
  W_{n} &=& s_{k}^{n-1}(0)s_{k}^{n-2}(0)s_{k}^{n-2}(2) \\
   &=& \Bigg(R_{1}G_{1}R_{2}G_{2}\ldots R_{n-1}G_{n-1}R_{n}[1, r_{n-1}]\Bigg) \\
&&\Bigg(R_{1}G_{1}R_{2}G_{2}\ldots R_{n-k}G_{n-k}R_{n-(k-1)}G_{n-(k-1)} \ldots R_{n-2}G_{n-2}R_{n-1}[1, r_{n-2}+1]\Bigg)s_{k}^{n-2}(2) \\
   &=& \Bigg(R_{1}G_{1}R_{2}G_{2}\ldots R_{n-1}G_{n-1}\Bigg)\Bigg(\underline{R_{n}[1, r_{n-1}]R_{1}G_{1}R_{2}G_{2}\ldots R_{n-k}}_{R_{n}}\Bigg) \\
&& G_{n-k}R_{n-(k-1)}G_{n-(k-1)} \ldots R_{n-2}G_{n-2}R_{n-1}[1, r_{n-2}+1]s_{k}^{n-2}(2) \\
   &=& R_{1}G_{1}R_{2}G_{2}\ldots R_{n-1}G_{n-1}\underline{R_{n}}G_{n-k}R_{n-(k-1)}G_{n-(k-1)} \ldots R_{n-2}G_{n-2}R_{n-1}[1, r_{n-2}+1]s_{k}^{n-2}(2) \\
   &=& R_{1}G_{1}R_{2}G_{2}\ldots R_{n-1}G_{n-1}R_{n}G_{n-k}\Bigg(\displaystyle \prod_{i=k-1}^{2}(R_{n-i}G_{n-i})\Bigg)R_{n-1}[1, r_{n-2}+1]s_{k}^{n-2}(2)
\end{eqnarray*}
Furthermore, according to Lemma \ref{l4.1}, we obtain:
\begin{equation*}
   W_{n-1} = R_{1}G_{1}R_{2}G_{2}\ldots R_{n}G_{n}R_{n+1}[1, r_{n}] .
\end{equation*}
Hence,
\begin{equation*}
    G_{n} = G_{n-k}\Bigg(\displaystyle \prod_{i=k-1}^{2}(R_{n-i}G_{n-i})\Bigg)R_{n-1}[1,r_{n-2}+1]s_{k}^{n-2}(0)\ast.
\end{equation*}
Since $R_{n+1}[1, r_{n}]$ is a suffix of $s_{k}^{n-2}(2).$\\
$\textbf{Case 2:}$ If $(n-1)\not\equiv 1 \pmod{k},$ then, $(n-2)\not\equiv 0 \pmod{k}$ and $(n-1) \not\equiv 0 \pmod{k}$ ou $(n-1) \equiv 0 \pmod{k}.$ \\
Therefore,
$$W_{n-1} = R_{1}G_{1}R_{2}G_{2}\ldots R_{n-1}G_{n-1}R_{n}[1, r_{n-1}+1]$$ or $$W_{n-1} = R_{1}G_{1}R_{2}G_{2}\ldots R_{n-1}G_{n-1}R_{n}[1, r_{n-1}+1]$$
 and
 $$W_{n-2} = R_{1}G_{1}R_{2}G_{2}\ldots R_{n-2}G_{n-2}R_{n-1}[1, r_{n-2}].$$
Then,
\begin{eqnarray*}
  W_{n} &=& s_{k}^{n-1}(0)s_{k}^{n-2}(0)s_{k}^{n-2}(2) \\
   &=& \Bigg(R_{1}G_{1}R_{2}G_{2}\ldots R_{n-1}G_{n-1}R_{n}[1, r_{n-1}+1]\Bigg) \\
&&\Bigg(R_{1}G_{1}R_{2}G_{2}\ldots R_{n-k}G_{n-k}R_{n-(k-1)}G_{n-(k-1)} \ldots R_{n-2}G_{n-2}R_{n-1}[1, r_{n-2}]\Bigg)s_{k}^{n-2}(2) \\
   &=& \Bigg(R_{1}G_{1}R_{2}G_{2}\ldots R_{n-1}G_{n-1}\Bigg)\Bigg(\underline{R_{n}[1, r_{n-1}+1]R_{1}G_{1}R_{2}G_{2}\ldots R_{n-k}}_{R_{n}}\Bigg) \\
&& G_{n-k}R_{n-(k-1)}G_{n-(k-1)} \ldots R_{n-2}G_{n-2}R_{n-1}[1, r_{n-2}]s_{k}^{n-2}(2) \\
   &=& R_{1}G_{1}R_{2}G_{2}\ldots R_{n-1}G_{n-1}\underline{R_{n}}G_{n-k}R_{n-(k-1)}G_{n-(k-1)} \ldots R_{n-2}G_{n-2}R_{n-1}[1, r_{n-2}]s_{k}^{n-2}(2) \\
   &=& R_{1}G_{1}R_{2}G_{2}\ldots R_{n-1}G_{n-1}R_{n}G_{n-k}\Bigg(\displaystyle \prod_{i=k-1}^{2}(R_{n-i}G_{n-i})\Bigg)R_{n-1}[1, r_{n-2}]s_{k}^{n-2}(2)
\end{eqnarray*}
Furthermore, according to Lemma \ref{l4.1}, we obtain
\begin{equation*}
   W_{n-1} = R_{1}G_{1}R_{2}G_{2}\ldots R_{n}G_{n}R_{n+1}[1, r_{n}+1].
\end{equation*}
Consequently,
\begin{equation*}
    G_{n} = G_{n-k}\Bigg(\displaystyle \prod_{i=k-1}^{2}(R_{n-i}G_{n-i})\Bigg)R_{n-1}[1,r_{n-2}]s_{k}^{n-2}(0)\ast.
\end{equation*}
Since $R_{n+1}[1, r_{n}+1]$ is a suffix of $s_{k}^{n-2}(2).$
\end{proof}

We now turn our attention to the lengths of the gap sequences $\{G_{n}\}_{n\geq1}$. Let $g_{n} := |G_{n}|$ for $n\geq 1$. It is straightforward to verify that
$g_{1}=0,\ g_{2} = 1$. For convenience, we also define $g_{0} = 0$.

\begin{corollary}
For $n\geq k+1,$ we obtain:
$$\left\{
  \begin{array}{ll}
    g_{n} = g_{n-k} + \displaystyle\sum_{i=k-1}^{2}~r_{n-i} + \displaystyle\sum_{i=k-1}^{2}~g_{n-i} + r_{n-2} + 1 + w_{n-2,2} - |\ast| \hspace*{0.9cm} if \ (n-1)\equiv 1 \pmod{k}\\
    g_{n} = g_{n-k} + \displaystyle\sum_{i=k-1}^{2}~r_{n-i} + \displaystyle\sum_{i=k-1}^{2}~g_{n-i} + r_{n-2} + w_{n-2,2} - |\ast| \hspace*{1.5cm} otherwise

  \end{array}
\right.$$
\end{corollary}

\begin{proposition}
For $n\geq k+1,$ the length of $G_{n}$ can be expressed by the following explicit formula:
\begin{equation*}
    g_{n} = (g_{k+1}+1)2^{n-(k+1)} - 1.
\end{equation*}

\end{proposition}
\begin{proof}
For all $n\geq k+1,$ the length $g_{n}$ resatisfies the following relation:
$$g_{n}=2g_{n-1}+1.$$
This relationship defines an arithmetic-geometric sequence. Thus, the general term of the sequence $(g_{n})_{n\geq k+1}$ is:
$$g_{n}=(g_{p}-l)\times 2^{n-p}+l,$$
where $l=\dfrac{1}{1-2}=-1$ and $g_{p}$, the initial term. It follows that:
$$g_{n}=(g_{k+1}+1)\times 2^{n-(k+1)}-1.$$
\end{proof}
\begin{example}
For $k = 4$, we consider the substitution $s_{4}$ given by: $$s_{4}: \left\{
                       \begin{array}{ll}
                         0 \mapsto 01  \\
                         1 \mapsto 02 \\
                         2 \mapsto 03 \\
                         3 \mapsto 00
                       \end{array}
                     \right.$$
The $4$-period-doubling word is the infinite word generated by $s_{4}$ from $0$. Thus:
\begin{align*}
\textbf{P}_{4}&=s^{\infty}_{4}(0)\\
&=0 ~ 1 ~ 0 ~ 2 ~ 010 ~~ 3 ~~ 010201 ~~ 000 ~~ 102010301020 ~~ 10101 ~~ 020103010201000102\ldots \cdots.
\end{align*}
The first few values, length of the gap sequence $\{G_{n}\}_{n\geq1},$ and the gap sequences of the $4$-period-doubling sequence are:
$$\left\{
  \begin{array}{ll}
    G_{1} = \epsilon \\
    G_{2} = 0 \\
    G_{3} = 010 \\
    G_{4} = 010201 \\
    G_{5} = 102010301020 \\
    G_{6} = 0201030102010001020103010 \\
    G_{7} = 010301020100010201030102010101020103010201000102010 \\
    \vdots
  \end{array}
\right. ~~ and ~~ \left\{
                    \begin{array}{ll}
                      g_{1} = 0 \\
                      g_{2} = 1 \\
                      g_{3} = 3 \\
                      g_{4} = 6 \\
                      g_{5} = 12 \\
                      g_{6} = 25 \\
                      g_{7} = 51 \\
                      \vdots
                    \end{array}
                  \right.$$

\end{example}

\section{Conclusion}
Factorization of infinite words is an essential tool for understanding their components. The factorization presented in this paper enhances our comprehension of the $k$-period-doubling substitution decomposition by analyzing its kernel words and the gap sequences between them.





\end{document}